\documentclass[11pt,letterpaper,reqno]{amsart}

\usepackage{amssymb}
\usepackage{amsmath}
\usepackage{amsthm}
\usepackage{bbm}
\usepackage{doi}

\usepackage{bookmark}
\usepackage{hyperref}
\hypersetup{pdfstartview={FitH}}

\numberwithin{equation}{section}

\newtheorem{theorem}{Theorem}
\newtheorem{lemma}[theorem]{Lemma}
\newtheorem{claim}{Claim}

\renewcommand{\leq}{\leqslant}
\renewcommand{\geq}{\geqslant}

\newcommand{\R}{\mathbb{R}}
\newcommand{\N}{\mathbb{N}}

\begin{document}

\title[On the set of points represented by harmonic subseries]{On the set of points represented by harmonic subseries}

\author[Vjekoslav Kova\v{c}]{Vjekoslav Kova\v{c}}

\address{Department of Mathematics, Faculty of Science, University of Zagreb, Bijeni\v{c}ka cesta 30, 10000 Zagreb, Croatia}

\email{vjekovac@math.hr}



\begin{abstract}
We help Alice play a certain ``convergence game'' against Bob and win the prize, which is a constructive solution to a problem by Erd\H{o}s and Graham, posed in their 1980 book on open questions in combinatorial number theory.
Namely, after several reductions using peculiar arithmetic identities, the game outcome shows that the set of points
\[ \Big(\sum_{n\in A}\frac{1}{n}, \sum_{n\in A}\frac{1}{n+1}, \sum_{n\in A}\frac{1}{n+2}\Big), \]
obtained as $A$ ranges over infinite sets of positive integers, has a non-empty interior. This generalizes a two-dimensional result by Erd\H{o}s and Straus.
\end{abstract}

\maketitle



\section{Introduction}
\subsection{The problem}
Paul Erd\H{o}s asked numerous questions on representations of numbers as finite or infinite sums of distinct unit fractions; see the book chapter \cite{Gra13} written by Graham or the problems stated in \cite[\S4,6,7]{EG80}. Many of these problems have been solved over the years and they have motivated the development of new techniques in number theory and combinatorics. 
The problem that we study in this note is distinctly real-valued and higher-dimensional and its solution will use very few ingredients from number theory.

Erd\H{o}s and Graham, in their 1980 monograph on open problems in combinatorial number theory \cite{EG80}, mentioned that
\begin{quote}
\emph{Erd\H{o}s and Straus [unpublished] proved that if one takes all sequences of integers $a_1,a_2,\ldots$, with $\sum_k 1/a_k < \infty$, then the set
\[ \Big\{ (x,y) \,:\, x=\sum_k\frac{1}{a_k},\ y=\sum_k\frac{1}{1+a_k} \Big\} \]
contains a [non-empty] open set;} \cite[p.~65]{EG80}.
\end{quote}
It is also understood from the context that the sequence $(a_k)$ is meant to be positive and strictly increasing. (Indeed, allowing the repetition of $a_k$'s would make this result and the following one significantly less challenging.) 
Erd\H{o}s and Graham then asked \cite[p.~65]{EG80}:
\begin{quote}
\emph{Is the same true in three (or more) dimensions, e.g., taking all $(x,y,z)$ with
\[ x=\sum_k\frac{1}{a_k},\ y=\sum_k\frac{1}{1+a_k},\ z=\sum_k\frac{1}{2+a_k} ? \]}
\end{quote}
We can give the positive answer to this three-dimensional problem, which has also been posed on Thomas Bloom's website \emph{Erd\H{o}s problems} \cite[Problem \#268]{EP}.
Let us state a theorem in the spirit of Bloom's equivalent reformulation of the problem.

\begin{theorem}\label{thm:main}
The set
\begin{equation}\label{eq:theset}
\bigg\{ \Big(\sum_{n\in A}\frac{1}{n}, \sum_{n\in A}\frac{1}{n+1}, \sum_{n\in A}\frac{1}{n+2}\Big) \,:\, A\subset\N \text{ is an infinite set with } \sum_{n\in A}\frac{1}{n}<\infty \bigg\} \subseteq \R^3
\end{equation}
has a non-empty interior.
\end{theorem}

Here $\N$ denotes the set of all positive integers.
Theorem \ref{thm:main} clearly generalizes the aforementioned unpublished two-dimensional result of Erd\H{o}s and Straus.

\subsection{The approach}
The proof of Theorem \ref{thm:main} will be reduced to a certain infinite strategic game of convergence.
Various topological games have been studied over the years \cite{Ber57,Tel87}, but we do not attempt to formulate or prove any general results in this field.
\emph{Alice} will be our hero who intends to represent every point of a non-empty open subset of $\R^3$ in the form 
\begin{equation}\label{eq:desiredform}
\Big(\sum_{n\in A}\frac{1}{n}, \sum_{n\in A}\frac{1}{n+1}, \sum_{n\in A}\frac{1}{n+2}\Big)
= \sum_{n\in A} \Big( \frac{1}{n}, \frac{1}{n+1}, \frac{1}{n+2} \Big)
\end{equation}
for an appropriate $A\subset\N$.
To relax the limitations coming from the arithmetic of $\N$, we will first interpret this three-dimensional series, after a linear change of variables given by the matrix \eqref{eq:thematrixm} below, as a perturbed series
\begin{equation}\label{eq:mainscheme}
\sum_{n} \bigg( \Big( \frac{1}{n}, \frac{2}{n^2}, \frac{2}{n^3} \Big) + O\Big(\frac{1}{n^4}\Big) \bigg)
\end{equation}
and several further reductions will follow.
Alice will then play a game against \emph{Bob}, who does not need to be Alice's real adversary, but simply plays the role of picking the most annoying choices for perturbation errors. A strategy for winning against Bob is a strategy that takes care of all possibilities for the error terms.
Since the task is not easy, Alice will first pass a training consisting of several easier games in Section \ref{sec:warmup}.
The reader who skips this warm-up section will still find the rigorous self-contained proof in Sections \ref{sec:arithmetic} and \ref{sec:mainproof}, only the motivation will be missing.

\subsection{The literature}
The author is not aware of an existing general result implying Theorem \ref{thm:main}, but there is a lot of literature on similar problems.
The study of the so-called \emph{achievement set}
\[ \Big\{ \sum_{n=1}^{\infty} \epsilon_n v_n \,:\, \epsilon_n\in\{0,1\} \text{ for every }n\in\N \Big\} \]
of a series $\sum_{n=1}^{\infty}v_n$ has been initiated by S\={o}ichi Kakeya \cite{Kak14a,Kak14b} more than a hundred years ago. 
The one-dimensional case, i.e., when $v_n\in(0,\infty)$, has later been resolved completely by Guthrie, Nymann, and S\'{a}enz \cite{GN88,NS00} and it is already non-trivial in the sense that there are four possible topological types of achievement sets; also see the survey \cite{BFP13} and the note \cite{Jon11}.
The same notion makes sense when $v_n$ are vectors in a Banach space, or when the original series $\sum_{n}v_n$ converges only conditionally.
Higher-dimensional theory seems to be mostly concerned with special cases and examples; see for instance \cite{BG15,BGM18,GM18} and references therein.
Fractal properties and dimensions of achievement sets have been studied by Mor\'{a}n \cite{Mor89,Mor94} and they can also be represented as ranges of (atomic) vector-valued measures \cite{LA08}.
It is conceivable that a more general result than Theorem \ref{thm:main} can be shown using less elementary techniques than ours, but we intentionally stick to the original question of Erd\H{o}s, Graham, and Straus, because of its elegant formulation and since the proof below is constructive and we find an explicit ball inside the set \eqref{eq:theset} in Section \ref{sec:concrete}. 

\subsection{Notation}
For two sequences of real numbers, $(x_n)$ and $(y_n)$, we write:
\begin{itemize}
\item $x_n = O(y_n)$ if there exists a constant $C\in(0,\infty)$ such that $|x_n|\leq C|y_n|$ for every $n\in\N$;
\item $x_n = \Omega(y_n)$ if there exists a constant $c\in(0,\infty)$ such that $|x_n|\geq c|y_n|$ for every $n\in\N$, i.e., if $y_n = O(x_n)$;
\item $x_n = \Theta(y_n)$ if both $x_n = O(y_n)$ and $y_n = O(x_n)$ hold.
\end{itemize}
We will also write
\[ (x_{n,1},x_{n,2},\ldots,x_{n,d}) = O(y_n) \]
if $x_{n,j} \in O(y_n)$ holds in each coordinate $1\leq j\leq d$.


\section{Warming up for the real game}
\label{sec:warmup}
Consider this section as a warm-up practice for Alice.
The first two games are rather trivial and devised for a single player only, but then things gradually become more interesting. 

\subsection{Game \#1}
The one-dimensional case of the Erd\H{o}s--Straus result is straightforward but it can still be illuminating.
Suppose that we want to represent a positive number $x$ as a sum of distinct terms of the harmonic series $\sum_{n=1}^{\infty}1/n$, i.e., as
\[ x=\sum_{n\in A}\frac{1}{n} \]
for some set $A\subset\N$.
Alice can achieve this easily by the following greedy algorithm. Set $x_0:=0$ and then define $(x_n)_{n=0}^{\infty}$ recursively as
\[ x_{n+1} := x_n + \begin{cases}
1/n & \quad\text{if } x_n + 1/n \leq x, \\
0 & \quad\text{otherwise}
\end{cases} \]
for every index $n\geq0$.
The sequence $(x_n)$ is increasing, so it converges to some number $x'\leq x$. Clearly $x' = x$, as otherwise we would get a contradiction as soon as the algorithm gets to the terms $1/n<x-x'$ and sufficiently many of these terms will eventually exceed the value $x'$ due to the divergence of the harmonic series.
This limit $x$ is now a (finite or infinite) sum of distinct unit fractions and we are done.

Exactly the same algorithm will represent every positive number $x$ as a sum of distinct unit fractions with odd denominators, i.e., as
\begin{equation}\label{eq:odddenom}
x=\sum_{n\in A}\frac{1}{2n-1} 
\end{equation}
for some $A\subset\N$. Interestingly, it is still open whether the latter greedy algorithm always stops after finite many steps for every rational number $x\in(0,1)$ with an odd denominator, even though it is known that it can be represented as a finite sum \eqref{eq:odddenom}. This is a conjecture by Sherman Stein; see \cite{Gra13}.

\subsection{Game \#2}
If the game is played with the series $\sum_{n=2}^{\infty}1/n^2$ instead, then Alice can use the same algorithm to represent every $x\in[0,\pi^2/6-1]$. Why does this greedy procedure work again?
Telescoping gives
\begin{equation}\label{eq:estik2}
\sum_{k=n+1}^{n+m}\frac{1}{k^2} > \sum_{k=n+1}^{n+m}\frac{1}{k(k+1)} = \sum_{k=n+1}^{n+m} \Big(\frac{1}{k}-\frac{1}{k+1}\Big) = \frac{1}{n+1} - \frac{1}{n+m+1},
\end{equation}
so the series tail $\sum_{k=n+1}^{\infty}1/k^2$ exceeds its term $1/n^2$ for $n\geq2$.
Series with this property go by different names in the literature; for instance Graham \cite{Gra64} would say that they have \emph{smoothly replaceable terms}.
Thus, $x_{n+1}=x_n$ in some step of the algorithm means that $x_n + 1/n^2 > x$, so the distance $x-x_n$ is smaller than $\sum_{k=n+1}^{\infty}1/k^2$ and it will surely be overcome by a sum of the later terms of the series.
Any Dirichlet series $\sum_{n}1/n^p$, for $p>1$, has the same property for sufficiently large indices $n$.

Things are not much different if we perturb the last series into something like
\begin{equation}\label{eq:subseriesugly}
\sum_{n=1}^{\infty} \Big( \frac{1}{n^2} + \frac{100\sin n}{(n^2+n+1)^2} \Big).
\end{equation}
For sufficiently large $n$, say $n\geq100$, the second summand is swiped into a $(1/100)$-th part of the first one, and Alice can use the same greedy algorithm again to represent every $x\in[0,0.0099]$ as a subseries of \eqref{eq:subseriesugly}.
The details are left to the reader.

\subsection{Game \#3}
Alice will need more flexibility to cope with Theorem \ref{thm:main}, such as allowing general perturbations of the series' terms, both those that are and those that are not used in the subseries representation of a particular number $x$.
Writing purely formally, Alice wants to represent every $x$ (from a carefully chosen interval) as
\[ x = \sum_{n} \bigg(\frac{\epsilon_n}{n^2} + O\Big(\frac{1}{n^4}\Big) \bigg), \]
where each $\epsilon_n$ is equal to $0$ or $1$.
Alice chooses the set of \emph{played} indices $n$ beforehand; they need to be sufficient to represent all considered numbers. Then, for a given $x$, Alice chooses which of the played indices $n$ have $\epsilon_n=1$, so that the series term is included in the subseries, and for which of the played indices we have $\epsilon_n=0$, and the term is not included in the subseries that converges to $x$.
The error terms $O(1/n^4)$ are arbitrary and they are present for every played index $n$.

We reformulate this problem rigorously into a two-player game by introducing Alice's opponent Bob, who is trying to ruin the intended limit $x$ by perturbing Alice's terms with $O(1/n^4)$ errors. Bob is unpredictable and Alice needs to cover all possibilities.
The following are the precise game rules.
\begin{itemize}
\item The game has infinitely many rounds $n=0,1,2,\ldots$. 
Before it begins, Alice decides which rounds $n$, corresponding to the terms of the series $\sum_n 1/n^2$, will be \emph{played} and which will be \emph{skipped}, e.g., the even ones, the prime ones, etc. There is no series term for $n=0$, so let us assume that round $0$ is always skipped.
\item The game starts with the \emph{initial value} $x_0=0$ and the \emph{target value} is some $x\in\R$.
\item If the $n$-th round is skipped, then the current value $x_n$ remains unchanged, i.e., we set $x_{n+1}:=x_n$.
\item If the $n$-th round is played, then Alice decides whether to add $1/n^2$ to the current value $x_n$ or not.
Afterwards, Bob further adds any number from the interval $[-C/n^4,C/n^4]$, producing $x_{n+1}$. For simplicity let $C=100$.
\item Alice's goal is to make the sequence $(x_n)$ converge to $x$, while Bob's goal is the opposite. 
\end{itemize}
Can Alice still win for every $x$ from some carefully chosen interval of target values?
The answer is: yes!

Postponing the discussion of the allowed values for $x$, we notice that the same greedy algorithm will not work for Alice anymore. 
Even anticipating Bob's immediate contribution and choosing to include the term $1/n^2$ only if the stronger criterion
\[ x_{n+1} = x_n + \frac{1}{n^2} + \frac{100}{n^4} \leq x \]
is satisfied will not help Alice either. Namely, it is possible that $x_{n+1}$ comes sufficiently close to $x$, so that Bob can keep adding positive error terms to make $x_{n+m}>x$ in some of the rounds played later and Alice will not be able to decrease the sequence terms. 
Then the sequence $(x_n)$ will not converge to $x$ and Bob will win.

Thus, Alice should make the sequence $(x_n)$ approach $x$ more cautiously in order to avoid possible overshoots coming from Bob's moves, either in the current round, or in any of the future rounds.
Luckily, from the $n$-th round on, Bob can only perturb the value by at most
{\allowdisplaybreaks\begin{align} 
\sum_{k=n}^{\infty} \frac{100}{k^4} 
& < \sum_{k=n}^{\infty} \frac{100}{(k-3)(k-2)(k-1)k} \nonumber \\
& = \frac{100}{3} \sum_{k=n}^{\infty} \Big( \frac{1}{(k-3)(k-2)(k-1)} - \frac{1}{(k-2)(k-1)k} \Big) \nonumber \\
& = \frac{100}{3(n-3)(n-2)(n-1)}, \label{eq:estik4}
\end{align}}
which is smaller than $1/n^2$ for, say, $n\geq 100$.
Thus, Alice chooses to play rounds $n\geq 100$ and decides to add the term $1/n^2$ to the current value $x_n$ only if
\[ x_n + \frac{3}{n^2} \leq x \]
and leave it as it is otherwise.
We see that Alice is not using a greedy algorithm and never allows $x_n$ to approach too quickly to its tentative limit $x$.

No matter what Bob ever does, Alice now at least knows that, after each played round $n$:
\begin{align}
\text{Case 1: } & x_n + \frac{3}{n^2} \leq x \implies \Big|x_{n+1} - x_{n} - \frac{1}{n^2}\Big| \leq \frac{100}{n^4}, \label{eq:2dcase1} \\
\text{Case 2: }& x_n + \frac{3}{n^2} > x \implies |x_{n+1} - x_{n}| \leq \frac{100}{n^4}. \label{eq:2dcase2}
\end{align}
In either case, $|x_{n+1}-x_n| \leq 101/n^2$, which easily implies that $(x_n)$ is a Cauchy sequence, 
\[ |x_{n+m}-x_n| \leq 101 \sum_{k=n}^{\infty} \frac{1}{k^2} < \frac{101}{n-1}, \]
so it converges to some number $x'\in\R$.
If both cases occur infinitely many times, then there exist indices $n(1)<n(2)<\cdots$ such that
\[ x_{n(k)} + \frac{3}{n(k)^2} \leq x \]
and indices $n'(1)<n'(2)<\cdots$ such that
\[ x_{n'(k)} + \frac{3}{n'(k)^2} > x \]
for every $k\in\N$.
Letting $k\to\infty$ we obtain
\[ \lim_{k\to\infty} x_{n(k)} \leq x \leq \lim_{k\to\infty} x_{n'(k)}, \]
but the whole sequence $(x_n)$ converges to $x'$. Thus $x'=x$ and Alice wins, but we still need to find values $x$ that guarantee that Cases 1 and 2 happen infinitely often.

As soon as $x\geq 4\cdot 10^{-5}$, we know that Case 1 will occur at least once, since otherwise we would arrive at a contradiction with
\begin{align*} 
x & \stackrel{\eqref{eq:2dcase2}}{<} \frac{3}{10^6} + x_{1000} 
= \frac{3}{10^6} + x_{1000} - x_{100} 
= \frac{3}{10^6} + \sum_{n=100}^{999} (x_{n+1}-x_n) \\
& \stackrel{\eqref{eq:2dcase2}}{\leq} \frac{3}{10^6} + \sum_{n=100}^{999} \frac{100}{n^4} 
\stackrel{\eqref{eq:estik4}}{<} \frac{3}{10^6} + \frac{100}{3\cdot 97\cdot 98\cdot 99} 
< 4\cdot 10^{-5}.
\end{align*}
After each occurrence of Case 1, say at the $n$-th round, we eventually end up in Case 1 again. Otherwise, besides \eqref{eq:2dcase1} we could also write
\begin{equation}\label{eq:2dcase3} 
x_{k} + \frac{3}{k^2} > x,\quad |x_{k+1} - x_{k}| \leq \frac{100}{k^4} 
\end{equation}
for every $k\geq n+1$, so
\begin{align*} 
x_n + \frac{3}{n^2} & \stackrel{\eqref{eq:2dcase1}}{\leq} x \stackrel{\eqref{eq:2dcase3}}{<} x_{2n} + \frac{3}{(2n)^2} 
= x_n + \frac{3}{4n^2} + \sum_{k=n}^{2n-1} (x_{k+1}-x_k) \\
& \stackrel{\eqref{eq:2dcase3}}{\leq} x_n + \frac{3}{4n^2} + \sum_{k=n}^{2n-1} \frac{100}{k^4} 
\stackrel{\eqref{eq:estik4}}{<} x_n + \frac{3}{4n^2} + \frac{100}{3(n-3)(n-2)(n-1)},
\end{align*}
which leads us to a contradiction for any $n\geq100$.
Next, if $x<3\cdot 10^{-4}$, then 
\[ x_{100} + \frac{3}{100^2} = 0 + 3\cdot 10^{-4} > x ,\]
so we will be in Case 2 already for the first played round, which is $n=100$.
After the occurrence of Case 2 at some played round $n$, we eventually end up in Case 2 again, as otherwise, besides \eqref{eq:2dcase2}, we would also have
\begin{equation}\label{eq:2dcase4} 
x_{k} + \frac{3}{k^2} \leq x,\quad \Big|x_{k+1} - x_{k}-\frac{1}{k^2}\Big| \leq \frac{100}{k^4} 
\end{equation}
for every $k\geq n+1$.
A contradiction is now obtained for any $n\geq100$ from
{\allowdisplaybreaks\begin{align*} 
x_n + \frac{3}{n^2} & \stackrel{\eqref{eq:2dcase2}}{>} x \stackrel{\eqref{eq:2dcase4}}{\geq} x_{2n} + \frac{3}{(2n)^2} 
= x_n + \frac{3}{4n^2} + \sum_{k=n}^{2n-1} (x_{k+1}-x_k) \\
& \stackrel{\eqref{eq:2dcase2},\eqref{eq:2dcase4}}{\geq} x_n + \frac{3}{4n^2} - \frac{100}{n^4} + \sum_{k=n+1}^{2n-1} \Big(\frac{1}{k^2}-\frac{100}{k^4}\Big) \\
& \stackrel{\eqref{eq:estik2},\eqref{eq:estik4}}{>} x_n + \frac{3}{4n^2} + \frac{1}{n+1} - \frac{1}{2n} - \frac{100}{3(n-3)(n-2)(n-1)}.
\end{align*}}
Therefore, Alice has the strategy to win the game for every $x\in[4\cdot 10^{-5},3\cdot 10^{-4})$.

We will refer to this game as \emph{playing with the series} denoted schematically as
\[ \sum_{n} \bigg(\frac{1}{n^2} + O\Big(\frac{1}{n^4}\Big) \bigg). \]
Other perturbed series have analogous games associated with them.

\subsection{Game \#4}
What will Alice do if Bob becomes more hostile and starts perturbing the terms by $O(1/n^3)$ instead, i.e., when they are playing with the series
\[ \sum_{n} \bigg(\frac{1}{n^2} + O\Big(\frac{1}{n^3}\Big) \bigg)? \]
Nothing easier than that!
Alice simply sparsifies the series, choosing to play only the rounds $n=k^2$ for $k=1,2,3,\ldots$. That way the game is effectively played with the series $\sum_k 1/k^4$, while Bob is perturbing the terms with the errors that are only $O(1/k^6)$. The tail $\sum_{l=k}^{\infty} C/l^6$ is now easily dominated by $C/k^5$, when it gets swallowed inside the main term $1/k^4$.

Various other sparsfications are possible, such as:
\begin{itemize}
\item choosing only the odd terms by substituting $n=2k-1$,
\item choosing some terms that are not divisible by $2$, $3$, or $5$ by substituting $n=30k+1$, etc.
\end{itemize}
All of them lead to minor modifications of game \#3.

\subsection{Game \#5}
Now consider the following two-dimensional series,
\begin{equation}\label{eq:seriesugly2}
\sum_{n=1}^{\infty} \Big( \frac{1-(-1)^n}{2n^2} + \frac{100\sin n}{(n^2+n+1)^2}, \frac{1+(-1)^n}{2n^2} + \frac{10\cos(n^2)}{2n^4-1} \Big).
\end{equation}
Let us ask Alice to represent every point $(x,y)$ from a non-degenerate disk in $\R^2$ as a sum of a subseries of \eqref{eq:seriesugly2}.
Lower order terms can be understood as inconvenient perturbations and \eqref{eq:seriesugly2} is just a particular instance of the perturbed series
\begin{equation}\label{eq:seriesugly3}
\sum_{n=1}^{\infty} \bigg( \Big( \frac{1-(-1)^n}{2n^2}, \frac{1+(-1)^n}{2n^2} \Big) + O\Big(\frac{1}{n^4}\Big) \bigg).
\end{equation}
Again, the game between Alice and Bob is played, starting from $(x_0,y_0)=(0,0)$ and producing the point $(x_{n+1},y_{n+1})$ from $(x_n,y_n)$ in its $n$-th round. Its rules are analogous to the rules of game \#3 and Alice's goal is to make this sequence of points converge to $(x,y)$.
It is natural to split \eqref{eq:seriesugly3} as
\[ \sum_{n\text{ odd}} \bigg( \Big( \frac{1}{n^2}, 0 \Big) + O\Big(\frac{1}{n^4}\Big) \bigg)
+ \sum_{n\text{ even}} \bigg( \Big( 0, \frac{1}{n^2} \Big) + O\Big(\frac{1}{n^4}\Big) \bigg). \]
Alice can now use the odd terms to ``move'' horizontally towards $x$ by simply playing with the series 
\[ \sum_{k} \bigg( \frac{1}{(2k-1)^2} + O\Big(\frac{1}{k^4}\Big) \bigg) \]
and, simultaneously, the even terms to ``move'' vertically towards $y$ by playing with 
\[ \sum_{k} \bigg( \frac{1}{(2k)^2} + O\Big(\frac{1}{k^4}\Big) \bigg). \]
That way the game splits into two one-dimensional games, played independently in each coordinate, except that the error terms mutually interact.
However, these error terms can be viewed as a single perturbation written as $O(1/k^4)$.
Our more flexible rules of game \#3 finally pay off. 

\subsection{Game \#6}
Alice is ready to prove the first nontrivial result. Let us tackle the unpublished two-dimensional result of Erd\H{o}s and Straus, mentioned in the introduction.
Since $(x,y)\mapsto (x,x-y)$ is an invertible linear transformation and
\[ \frac{1}{n} - \frac{1}{n+1} = \frac{1}{n(n+1)}, \]
we actually need to show that 
\[ \bigg\{ \Big(\sum_{n\in A}\frac{1}{n}, \sum_{n\in A}\frac{1}{n(n+1)} \Big) \in\R^2 \,:\, A\subset\N \text{ is an infinite set with } \sum_{n\in A}\frac{1}{n}<\infty \bigg\} \]
has a non-empty interior.
Also,
\[ \frac{1}{n(n+1)} = \frac{1}{n^2}\frac{1}{1+1/n} = \frac{1}{n^2} \bigg( 1 + O\Big(\frac{1}{n}\Big) \bigg), \] 
so Alice actually has to represent points (from some non-degenerate disk) as subseries' sums of the two-dimensional perturbed series
\[ \sum_n \bigg( \Big( \frac{1}{n}, \frac{1}{n^2} \Big) + O\Big(\frac{1}{n^3}\Big) \bigg). \]
The problem here, unlike with game \#5, is that the main term $(1/n,1/n^2)$ does not allow us to move along the vertical axis without producing a large perturbation error in the first coordinate, namely $O(1/n)$.
It is time that Alice learns one last trick!

Our hero could also consider differences of the series' terms, hoping to achieve a similar dynamics as in game \#5, but starting from a point $(x_0,y_0)$, which is no longer the origin $(0,0)$. Let us give a concrete example.
Dividing the Pythagorean identity $4^2 + 3^2 = 5^2$ by $60^2$, we obtain
\[ \frac{1}{15^2} + \frac{1}{20^2} = \frac{1}{12^2}. \]
Thus
\[ \underbrace{\Big( \frac{1}{15n}, \frac{1}{(15n)^2} \Big) + \Big( \frac{1}{20n}, \frac{1}{(20n)^2} \Big)}_{\text{add terms}} - \underbrace{\Big( \frac{1}{12n}, \frac{1}{(12n)^2} \Big)}_{\text{remove term}} 
= \Big( \frac{1}{30n}, 0 \Big) \]
and Alice can use this difference to ``move to the right'' by $1/(30n)$.
Similarly, starting with 
\[ \frac{1}{2} = \frac{1}{3} + \frac{1}{6}, \]
we easily get
\[ \underbrace{\Big( \frac{1}{2n}, \frac{1}{(2n)^2} \Big)}_{\text{add term}} - \underbrace{\Big( \frac{1}{3n}, \frac{1}{(3n)^2} \Big) - \Big( \frac{1}{6n}, \frac{1}{(6n)^2} \Big)}_{\text{remove terms}} 
= \Big( 0, \frac{1}{9n^2} \Big), \]
so Alice can also ``move up'' by $1/(9n^2)$.
Purelly formally, it would be desirable to start the game with
\begin{equation}\label{eq:startx0y0}
(x_0,y_0) = \sum_{n} \bigg( \Big( \frac{1}{3n}, \frac{1}{3n(3n+1)} \Big) + \Big( \frac{1}{6n}, \frac{1}{6n(6n+1)} \Big) + \Big( \frac{1}{12n}, \frac{1}{12n(12n+1)} \Big) \bigg),
\end{equation}
i.e., initially include the terms that Alice might decide to remove, and then represent, after Alice's win,
\begin{align*}
(x,y) & = (x_0,y_0) \\
& + \sum_{n} \epsilon_n \bigg( \Big( \frac{1}{15n}, \frac{1}{15n(15n+1)} \Big) + \Big( \frac{1}{20n}, \frac{1}{20n(20n+1)} \Big) - \Big( \frac{1}{12n}, \frac{1}{12n(12n+1)} \Big) \bigg) \\
& + \sum_{n} \epsilon'_n \bigg( \Big( \frac{1}{2n}, \frac{1}{2n(2n+1)} \Big) - \Big( \frac{1}{3n}, \frac{1}{3n(3n+1)} \Big) - \Big( \frac{1}{6n}, \frac{1}{6n(6n+1)} \Big) \bigg)
\end{align*}
for some $\epsilon_n,\epsilon'_n\in\{0,1\}$, which would then give
{\allowdisplaybreaks\begin{align*}
(x,y) = \sum_{n} \bigg( & \epsilon'_n \Big( \frac{1}{2n}, \frac{1}{2n(2n+1)} \Big) + (1-\epsilon'_n) \Big( \frac{1}{3n}, \frac{1}{3n(3n+1)} \Big) \\
& + (1-\epsilon'_n) \Big( \frac{1}{6n}, \frac{1}{6n(6n+1)} \Big)  + (1-\epsilon_n) \Big( \frac{1}{12n}, \frac{1}{12n(12n+1)} \Big) \\
& + \epsilon_n \Big( \frac{1}{15n}, \frac{1}{15n(15n+1)} \Big) + \epsilon_n \Big( \frac{1}{20n}, \frac{1}{20n(20n+1)} \Big) \bigg),
\end{align*}}
but there are two complications.
First, the series \eqref{eq:startx0y0} does not even converge.
Second, the numbers $2n,3n,6n,12n,15n,20n$ can repeat as $n$ varies, while a subseries should pick each term at most once.
Both of these difficulties are easily resolved by substituting $n=30k^2+1$ for $k=1,2,\ldots$.
For the purpose of game \#6 it is also convenient to allow Alice to permute the rounds arbitrarily, which is ok as long as we are dealing with absolutely convergent series.

The reader could now fill in the details to obtain a full proof of the Erd\H{o}s--Straus result. We do not do that here, since with just a few more technical details we can also give a complete proof of the three-dimensional result stated in Theorem \ref{thm:main}.


\section{An arithmetic lemma}
\label{sec:arithmetic}
We now make a final preparation for the proof of the main theorem.
The arithmetic of $\N$ will play a minimal role and it is encoded in the following auxiliary result, which does a similar initial reduction as the one we used in game \#6.

\begin{lemma}\label{lm:arithmetic}
There exist a matrix $M\in\textup{GL}(3,\R)$, mutually disjoint finite sets $S_1$, $S_2$, $S_3$, $T_1$, $T_2$, $T_3\subset\N$, and constants $c_1,c_2,c_3\in(0,\infty)$ such that
\begin{equation}\label{eq:matrix}
\Big( \sum_{a\in S_j} - \sum_{a\in T_j} \Big) M \begin{pmatrix} 1/(a n) \\ 1/(a n + 1) \\ 1/(a n + 2) \end{pmatrix}
= \frac{c_j}{n^j}\mathbbm{e}_j + O\Big(\frac{1}{n^{4}}\Big)
\end{equation}
for $1\leq j\leq 3$.
\end{lemma}

Here $\mathbbm{e}_1,\mathbbm{e}_2,\mathbbm{e}_3$ denotes the standard basis of $\R^3$.
From now on, it will be more convenient to write points and vectors of $\R^3$ as $3\times 1$ columns.

\begin{proof}[Proof of Lemma \ref{lm:arithmetic}]
Property \eqref{eq:matrix} relies on the expansions
\[ \frac{1}{a n + k - 1} = \frac{1}{a n} - \frac{k-1}{a^2 n^2} + \frac{(k-1)^2}{a^3 n^3} + O\Big(\frac{1}{n^4}\Big), \]
for $1\leq k\leq 3$, where the error terms are allowed to depend on a parameter $a\in\N$. 
If we choose
\begin{equation}\label{eq:thematrixm}
M = \begin{pmatrix} 1 & 0 & 0 \\ 3 & -4 & 1 \\ 1 & -2 & 1 \end{pmatrix}, 
\end{equation}
then
\[ M \begin{pmatrix} \frac{1}{a n} \\[1mm] \frac{1}{a n + 1} \\[1mm] \frac{1}{a n + 2} \end{pmatrix}
= \begin{pmatrix} \frac{1}{a n} + O(\frac{1}{n^4}) \\[1mm] \frac{2}{a^2 n^2} + O(\frac{1}{n^4}) \\[1mm] \frac{2}{a^3 n^3} + O(\frac{1}{n^4}) \end{pmatrix}. \]
Thus, to construct (say) $S_3$ and $T_3$, we need to make sure that
\[ \sum_{a\in S_3} \frac{1}{a} = \sum_{a\in T_3} \frac{1}{a},\qquad \sum_{a\in S_3} \frac{1}{a^2} = \sum_{a\in T_3} \frac{1}{a^2},\qquad \sum_{a\in S_3} \frac{1}{a^3} \neq \sum_{a\in T_3} \frac{1}{a^3}. \]
The only remaining ingredients of the proof are then elementary identities
{\allowdisplaybreaks\begin{align*}
\frac{1}{10}+\frac{1}{30}+\frac{1}{60} & = \frac{1}{12}+\frac{1}{15}, \\
\frac{1}{10^2}+\frac{1}{30^2}+\frac{1}{60^2} & = \frac{1}{12^2}+\frac{1}{15^2}, \\
\frac{1}{16}+\frac{1}{20}+\frac{1}{240} & = \frac{1}{15}+\frac{1}{24}+\frac{1}{120}, \\
\frac{1}{16^3}+\frac{1}{20^3}+\frac{1}{240^3} & = \frac{1}{15^3}+\frac{1}{24^3}+\frac{1}{120^3}, \\
\frac{1}{45^2}+\frac{1}{72^2}+\frac{1}{144^2}+\frac{1}{160^2}+\frac{1}{432^2}+\frac{1}{480^2} & =
\frac{1}{48^2}+\frac{1}{60^2}+\frac{1}{120^2}+\frac{1}{720^2}+\frac{1}{1440^2}+\frac{1}{4320^2}, \\
\frac{1}{45^3}+\frac{1}{72^3}+\frac{1}{144^3}+\frac{1}{160^3}+\frac{1}{432^3}+\frac{1}{480^3} & = 
\frac{1}{48^3}+\frac{1}{60^3}+\frac{1}{120^3}+\frac{1}{720^3}+\frac{1}{1440^3}+\frac{1}{4320^3}. 
\end{align*}}
Straightforward (preferably computer assisted) computation verifies these identities and thus also property \eqref{eq:matrix} with
{\allowdisplaybreaks\begin{align*}
S_1 = \{45, 72, 144, 160, 432, 480\}, & \qquad T_1 = \{48, 60, 120, 720, 1440, 4320\}, \\
S_2 = 11\cdot \{16, 20, 240\}, & \qquad T_2 = 11\cdot \{15, 24, 120\}, \\
S_3 = 7\cdot \{10, 30, 60\}, & \qquad T_3 = 7\cdot \{12, 15\}
\end{align*}}
and
{\allowdisplaybreaks\begin{align*}
c_1 & = \frac{1}{45}+\frac{1}{72}+\frac{1}{144}+\frac{1}{160}+\frac{1}{432}+\frac{1}{480} 
-\frac{1}{48}-\frac{1}{60}-\frac{1}{120}-\frac{1}{720}-\frac{1}{1440}-\frac{1}{4320} = \frac{1}{180}, \\ 
c_2 & = \frac{2}{11^2} \Big( \frac{1}{16^2}+\frac{1}{20^2}+\frac{1}{240^2} -\frac{1}{15^2}-\frac{1}{24^2}-\frac{1}{120^2} \Big) = \frac{1}{348480}, \\
c_3 & = \frac{2}{7^3} \Big( \frac{1}{10^3}+\frac{1}{30^3}+\frac{1}{60^3} -\frac{1}{12^3}-\frac{1}{15^3} \Big) = \frac{1}{1029000}. 
\end{align*}}
The author used Mathematica \cite{Mathematica} for this purpose.
Note the factors $7$ and $11$, the purpose of which is simply to make all numbers in $S_j,T_j$ mutually different.
The matrix $M$ is regular since $\det M=-2$.
\end{proof}

Our application of Lemma \ref{lm:arithmetic} in the proof below does not use that the sets $S_j$ and $T_j$ are finite: they could merely have finitely many prime factors.


\section{Proof of Theorem \ref{thm:main}}
\label{sec:mainproof}
Let $m$ be the product of all prime factors of the numbers in the set 
\[ U := S_1\cup S_2\cup S_3\cup T_1\cup T_2\cup T_3. \]
The proof of Lemma \ref{lm:arithmetic} actually gives 
\[ m = 2\cdot 3\cdot 5\cdot 7\cdot 11 = 2310. \] 
Every positive integer has at most one representation as $a(k^2 m+1)$ for $a\in U$ and $k\in\N$.
The series
\[ \sum_{\substack{a\in U\\k\in\N}} \frac{1}{a(k^2 m+1)},\quad \sum_{\substack{a\in U\\k\in\N}} \frac{1}{a(k^2 m+1)+1},\quad \sum_{\substack{a\in U\\k\in\N}} \frac{1}{a(k^2 m+1)+2} \]
converge, so even after multiplying their terms arbitrarily by $-1$, $0$, or $1$, we obtain absolutely convergent series, the terms of which can be freely permuted and grouped.
Let $C\in(0,\infty)$ be a constant such that each component of the vector
\begin{equation}\label{eq:vecexpdiff} 
\Big( \sum_{a\in S_j} - \sum_{a\in T_j} \Big) M \begin{pmatrix} 1/(a n) \\ 1/(a n + 1) \\ 1/(a n + 2) \end{pmatrix}
- \frac{c_j}{n^j}\mathbbm{e}_j 
\end{equation}
is at most $C/n^4$ in the absolute value, for every $n\in\N$ and every $1\leq j\leq 3$; it exists by Lemma \ref{lm:arithmetic}.
Next, let $K$ be a sufficiently large positive integer such that
{\allowdisplaybreaks\begin{align}
\underbrace{\sum_{l=k}^{\infty} \frac{3C}{(l^2 m+1)^4}}_{\Theta(k^{-7})} & < \underbrace{\frac{c_j}{(k^2 m+1)^j}}_{\Omega(k^{-6})}, \label{eq:propK1} \\
\underbrace{\sum_{l=k+1}^{\infty} \frac{c_j}{(l^2 m+1)^j}}_{\Theta(k^{-2j+1})} & > \underbrace{\frac{4c_j}{(k^2 m+1)^j}}_{\Theta(k^{-2j})} \label{eq:propK2}
\end{align}}
both hold for every $k\geq K$ and $1\leq j\leq 3$.

Let the real game begin!
Alice and Bob will start with the initial point
\[ p = \begin{pmatrix} p_1 \\ p_2 \\ p_3 \end{pmatrix} := \sum_{l=K}^{\infty} \sum_{j=1}^{3} \sum_{a\in T_j} M \begin{pmatrix} \frac{1}{a(l^2 m+1)} \\[1mm] \frac{1}{a(l^2 m+1)+1} \\[1mm] \frac{1}{a(l^2 m+1)+2} \end{pmatrix} \in \R^3 . \]
They are playing the game schematically represented by the perturbed series \eqref{eq:mainscheme}.
Alice chooses to play the rounds $n=a(k^2 m+1)$ for $a\in U$ and $k\geq K$, but the order of rounds is permuted and arranged according to $k$ only.
In each step Alice can increase the current point in some of the coordinates $j=1,2,3$ respectively by $c_j/(k^2 m+1)^j$, while Bob can perturb each of the coordinates at most by $C/(k^2 m+1)^4$ in the absolute value.

Alice claims a winning strategy for every target point
\begin{equation}\label{eq:choiceq}
q = \begin{pmatrix} q_1 \\ q_2 \\ q_3 \end{pmatrix} \in \mathcal{Q} := p + \prod_{j=1}^{3} \Big[\frac{c_j}{(K^2 m+1)^j}, \frac{2c_j}{(K^2 m+1)^j}\Big] 
\end{equation}
and will construct a sequence of points $(x_k)_{k=K}^{\infty}$,
\[ x_k = \begin{pmatrix} x_{k,1} \\ x_{k,2} \\ x_{k,3} \end{pmatrix} \in \R^3, \]
that converges to the point $q$ as follows.
First set $x_K := p$ and then define recursively
\[ x_{k+1} := x_k + \sum_{j=1}^{3} \epsilon_{k,j} \Big( \sum_{a\in S_j} - \sum_{a\in T_j} \Big) M \begin{pmatrix} \frac{1}{a(k^2 m+1)} \\[1mm] \frac{1}{a(k^2 m+1)+1} \\[1mm] \frac{1}{a(k^2 m+1)+2} \end{pmatrix} \]
for every $k\geq K$, where the coefficients $\epsilon_{k,j}\in\{0,1\}$ are determined from
\begin{align*}
\epsilon_{k,j} = 1 & \quad\text{if } x_{k,j} + \frac{3c_j}{(k^2 m+1)^j} \leq q_j, \\
\epsilon_{k,j} = 0 & \quad\text{otherwise}
\end{align*}
for $1\leq j\leq 3$.
From this recurrence relation and the definition of the constant $C$, we have
\begin{align}
\text{if } \epsilon_{k,j} = 1, & \quad\text{then } \Big| x_{k+1,j} - x_{k,j} - \frac{c_j}{(k^2 m+1)^j} \Big| \leq \frac{3C}{(k^2 m+1)^4}, \label{eq:rec1} \\
\text{if } \epsilon_{k,j} = 0, & \quad\text{then } \big| x_{k+1,j} - x_{k,j} \big| \leq \frac{3C}{(k^2 m+1)^4}. \label{eq:rec2}
\end{align}
We need two auxiliary claims in the proof that $(x_k)_{k=K}^{\infty}$ really converges to $q$, 
but a patient reader will be glad to notice that the whole proof is not more difficult than our discussion of game \#3.

\begin{claim}\label{cl:claim1}
For each $1\leq j\leq3$ there exist infinitely many $k\geq K$ such that $\epsilon_{k,j}=1$.
\end{claim}

\begin{proof}[Proof of Claim \ref{cl:claim1}]
We can immediately rule out the possibility $\epsilon_{k,j}=0$ for each $k\geq K$, as this would imply
\begin{equation}\label{eq:aux31}
x_{N,j} + \frac{3c_j}{(N^2 m+1)^j} > q_j \quad\text{for } N\geq K 
\end{equation}
and
\[ x_{k+1,j} \stackrel{\eqref{eq:rec2}}{\leq} x_{k,j} + \frac{3C}{(k^2 m+1)^4} \quad\text{for } k\geq K, \]
so that
\begin{equation}\label{eq:aux32}
x_{N,j} \leq p_j + \sum_{k=K}^{N-1}\frac{3C}{(k^2 m+1)^4} 
\end{equation}
for $N>K$.
A combination of \eqref{eq:aux31} and \eqref{eq:aux32} then yields
\[ q_j \leq p_j + \sum_{k=K}^{N-1}\frac{3C}{(k^2 m+1)^4} + \frac{3c_j}{(N^2 m+1)^j} \]
and letting $N\to\infty$ we conclude
\[ q_j \leq p_j + \sum_{k=K}^{\infty}\frac{3C}{(k^2 m+1)^4} 
\stackrel{\eqref{eq:propK1}}{<} p_j + \frac{c_j}{(K^2 m+1)^j}, \]
which contradicts the choice of $q$, namely \eqref{eq:choiceq}.

Similarly, suppose that there exists the largest integer $L\geq K$ such that $\epsilon_{L,j}=1$.
This means
\[ x_{L,j} + \frac{3c_j}{(L^2 m+1)^j} \leq q_j \]
and
\[ x_{k,j} + \frac{3c_j}{(k^2 m+1)^j} > q_j \quad\text{for } k\geq L+1, \]
implying
\begin{equation}\label{eq:aux11}
x_{N,j} - x_{L,j} \geq \frac{3c_j}{(L^2 m+1)^j} - \frac{3c_j}{(N^2 m+1)^j} 
\end{equation}
for every $N>L$.
As opposed to that, the recurrence relation gives
\[ |x_{L+1,j} - x_{L,j}| \stackrel{\eqref{eq:rec1}}{\leq} \frac{c_j}{(L^2 m+1)^j} + \frac{3C}{(L^2 m+1)^4} \]
and
\[ |x_{k+1,j} - x_{k,j}| \stackrel{\eqref{eq:rec2}}{\leq} \frac{3C}{(k^2 m+1)^4} \quad\text{for } k\geq L+1, \]
which yields
\begin{equation}\label{eq:aux12}
|x_{N,j} - x_{L,j}| \leq \frac{c_j}{(L^2 m+1)^j} + \sum_{k=L}^{N-1} \frac{3C}{(k^2 m+1)^4}
\stackrel{\eqref{eq:propK1}}{\leq} \frac{2c_j}{(L^2 m+1)^j}.
\end{equation}
Note that \eqref{eq:aux11} and \eqref{eq:aux12} together give
\[ \frac{c_j}{(L^2 m+1)^j} \leq \frac{3c_j}{(N^2 m+1)^j}. \]
We arrive at a clear contradiction in the limit as $N\to\infty$, so the claim follows.
\end{proof}

\begin{claim}\label{cl:claim2}
For each $1\leq j\leq3$ there exist infinitely many $k\geq K$ such that $\epsilon_{k,j}=0$.
\end{claim}

\begin{proof}[Proof of Claim \ref{cl:claim2}]
We certainly have $\epsilon_{K,j}=0$ because
\[ x_{K,j} + \frac{3c_j}{(K^2 m+1)^j} 
= p_j + \frac{3c_j}{(K^2 m+1)^j} > q_j \]
by \eqref{eq:choiceq}.
Suppose that the claim fails and that $L\geq K$ is the largest integer such that $\epsilon_{L,j}=0$.
We can read this as
\[ x_{L,j} + \frac{3c_j}{(L^2 m+1)^j} > q_j \]
and
\[ x_{k,j} + \frac{3c_j}{(k^2 m+1)^j} \leq q_j \quad\text{for } k\geq L+1. \]
Thus,
\begin{equation}\label{eq:aux21}
x_{N,j} - x_{L,j} \leq \frac{3c_j}{(L^2 m+1)^j}
\end{equation}
for every $N>L$.
The sequence $(x_k)$ has been defined in a way that
\[ x_{L+1,j} \stackrel{\eqref{eq:rec2}}{\geq} x_{L,j} - \frac{3C}{(L^2 m+1)^4} \]
and
\[ x_{k+1,j} \stackrel{\eqref{eq:rec1}}{\geq} x_{k,j} + \frac{c_j}{(k^2 m+1)^j} - \frac{3C}{(k^2 m+1)^4} \quad\text{for } k\geq L+1, \]
implying
\begin{align} 
x_{N,j} - x_{L,j} & \geq \sum_{k=L+1}^{N-1} \frac{c_j}{(k^2 m+1)^j} - \sum_{k=L}^{N-1} \frac{3C}{(k^2 m+1)^4} \nonumber \\
& \stackrel{\eqref{eq:propK1}}{\geq} \sum_{k=L+1}^{N-1} \frac{c_j}{(k^2 m+1)^j} - \frac{c_j}{(L^2 m+1)^j}. \label{eq:aux22}
\end{align}
Combining \eqref{eq:aux21} and \eqref{eq:aux22} we get
\[ \sum_{l=L+1}^{N-1} \frac{c_j}{(l^2 m+1)^j} \leq \frac{4c_j}{(L^2 m+1)^j}, \]
so letting $N\to\infty$ we obtain an estimate that contradicts \eqref{eq:propK2}.
\end{proof}

Now we can return to the proof of Theorem \ref{thm:main}.
From $|x_{k+1}-x_k| = O(1/k^2)$ it follows that $(x_k)$ is a Cauchy sequence in $\R^3$, so it converges.
By Claim \ref{cl:claim1} there exist infinitely many indices $k(j,1)<k(j,2)<\cdots$ such that
\[ x_{k(j,t),j} + \frac{3c_j}{(k(j,t)^2 m+1)^j} \leq q_j, \]
which gives
\[ \lim_{k\to\infty} x_{k,j} = \lim_{t\to\infty} x_{k(j,t),j} \leq q_j. \]
From Claim \ref{cl:claim2} we have infinitely many indices $k'(j,1)<k'(j,2)<\cdots$ such that
\[ x_{k'(j,t),j} + \frac{3c_j}{(k'(j,t)^2 m+1)^j} > q_j, \]
which implies
\[ \lim_{k\to\infty} x_{k,j} = \lim_{t\to\infty} x_{k'(j,t),j} \geq q_j. \]
It follows that $(x_k)$ converges precisely to the point $q$.
Alice wins for every $q$ from the rectangular box \eqref{eq:choiceq}.

By the definition of the sequence $(x_k)$ we have
\[ \lim_{k\to\infty} M^{-1} x_k 
= \sum_{\substack{k\geq K\\ 1\leq j\leq 3\\ \epsilon_{k,j}=1}} \sum_{a\in S_j} \begin{pmatrix} \frac{1}{a(k^2 m+1)} \\[1mm] \frac{1}{a(k^2 m+1)+1} \\[1mm] \frac{1}{a(k^2 m+1)+2} \end{pmatrix}
+ \sum_{\substack{k\geq K\\ 1\leq j\leq 3\\ \epsilon_{k,j}=0}} \sum_{a\in T_j} \begin{pmatrix} \frac{1}{a(k^2 m+1)} \\[1mm] \frac{1}{a(k^2 m+1)+1} \\[1mm] \frac{1}{a(k^2 m+1)+2} \end{pmatrix}, \]
compare with the discussion of game \#6.
Therefore, for every $q$ as in \eqref{eq:choiceq}, the point $M^{-1}q$ is of the desired form \eqref{eq:desiredform} for the set
\[ A := \bigcup_{\substack{k\geq K\\ 1\leq j\leq 3\\ \epsilon_{k,j}=1}} \{ a(k^2 m+1) : a\in S_j \}
\cup \bigcup_{\substack{k\geq K\\ 1\leq j\leq 3\\ \epsilon_{k,j}=0}} \{ a(k^2 m+1) : a\in T_j \}. \]
We are done since $M^{-1}\mathcal{Q}$ is a non-degenerate parallelepiped, so it has a non-empty interior.
The proof of Theorem \ref{thm:main} is complete.


\section{A concrete open set}
\label{sec:concrete}
Since the set of points \eqref{eq:theset} defined in Theorem \ref{thm:main} is a concrete subset of $\R^3$, we should be able to find an explicit open ball inside it. Luckily, our proof is constructive and we can compute the explicit values easily using Mathematica \cite{Mathematica}.

Some sequence monotonicity observations give that the constant $C$ only needs to dominate in the absolute value the corresponding limits as $n\to\infty$ of the coordinates of \eqref{eq:vecexpdiff} multiplied by $n^4$. This gives the optimal choice:
\[ C = \frac{8833}{100776960000} = 8.7649\ldots \cdot 10^{-8}. \]
Next, we need to find $K$ such that \eqref{eq:propK1} and \eqref{eq:propK2} hold for $k\geq K$ and $1\leq j\leq 3$.
By brute force verification for $k\leq 25$ and using
\[ \sum_{l=k}^{\infty} \frac{1}{(l^2 m+1)^4} < \int_{k-1}^{\infty} \frac{\textup{d}x}{(x^2 m+1)^4}, \quad
\sum_{l=k+1}^{\infty} \frac{1}{(l^2 m+1)^j} > \int_{k+1}^{\infty} \frac{\textup{d}x}{(x^2 m+1)^j} \]
for $k>25$, we see that we can take $K=14$.

All this is sufficient to compute (up to a desired precision) the point $p$ and the rectangular box appearing in the formula \eqref{eq:choiceq}.
Taking the smallest ball inside the box, transforming it by $M^{-1}$ into an ellipsoid, and estimating the singular values of $M^{-1}$ to find the smallest axis of that ellipsoid, we conclude that \eqref{eq:theset} contains a ball of radius $10^{-24}$ around the point
\[ \begin{pmatrix} 
2.58842922071730660744793282484 \\ 
2.58842919367011667177209233699 \\ 
2.58842916662292797961469594496
\end{pmatrix} \cdot 10^{-6}. \]


\section*{Acknowledgments}
This work was supported in part by the Croatian Science Foundation under the project HRZZ-IP-2022-10-5116 (FANAP).
The author is grateful to Thomas Bloom for a suggestion to find an explicit non-empty open set in the interior.


\bibliography{harmonicsubseries}
\bibliographystyle{plainurl}

\end{document}